\documentclass[11pt]{article}
\usepackage[toc,page]{appendix}

\usepackage{enumerate,amsmath,amsthm,latexsym,amssymb}
\usepackage{bbm}
\usepackage{enumerate}
\usepackage{algorithm,enumerate}
\usepackage{algpseudocode}
\usepackage{hyperref}
\allowdisplaybreaks

\usepackage{color}\usepackage{graphicx}

\def\cP{\mathcal{P}}
\def\cB{\mathcal{B}}

\def\e{\epsilon}

\renewcommand{\Pr}{\operatorname{\bf Pr}}
\newcommand\bfrac[2]{\left(\frac{#1}{#2}\right)}

\newtheorem{theorem}{Theorem}[section]

\newtheorem{lemma}[theorem]{Lemma}
\newtheorem{remark}[theorem]{Ramark}

\date{}
\title{Constructing Hamilton cycles and perfect matchings efficiently}
\author{Michael Anastos\footnote{
Institute of Science and Technology Austria, 
Klosterneurburg 3400, Austria. Email: michael.anastos@ist.ac.at}
}

\begin{document}
\maketitle
\begin{abstract}
Let $\e>0$. We consider the problem of constructing a Hamiltonian graph with $(1+\e)n$ edges in the following controlled random graph process. Starting with the empty graph on $[n]$, at each round a set of $K=K(n)$ edges is presented, chosen uniformly at random from the missing ones (or from the ones that have not been presented yet), and we are asked to choose at most one of them and add it to the current graph. We show that in this process one can build a Hamiltonian graph with at most $(1+\e)n$ edges in $(1+\epsilon)(1+(\log n)/2K) n$ rounds  w.h.p. The case $K=1$ implies that w.h.p. one can build a Hamiltonian graph by choosing $(1+\e)n$ edges in an on-line fashion as they appear along the first $(0.5+\epsilon)n\log n$ steps of the random graph process, this refutes a conjecture of Frieze, Krivelevich and Michaeli. The case $K=\Theta(\log n)$ implies that the Hamiltonicity threshold of the corresponding Achlioptas process is at most $(1+\epsilon)(1+(\log n)/2K) n$. This matches the $(1-\epsilon)(1+(\log n)/2K) n$ lower bound due to Krivelevich, Lubetzky and Sudakov and resolves the problem of determining the Hamiltonicity threshold of the Achlioptas process with $K=\Theta(\log n)$.

We also show that in the above process w.h.p. one can construct a graph $G$ that spans a matching of size $\lfloor V(G)/2) \rfloor$, with $(0.5+\e)n$ edges, within $(1+\epsilon)(0.5+(\log n)/2K) n$ rounds.

Our proof relies on a robust Hamiltonicity property of the strong $4$-core of the binomial random graph which we use as a black-box. This property allows it to absorb paths covering vertices outside the strong $4$-core into a cycle.

\end{abstract}

\section{Introduction}
Let $G_0,G_1,...,G_N$, $N=\binom{n}{2}$ be the random graph process. That is, $G_0$ is the empty graph on $[n]$ and $G_{i+1}$ is formed by adding to $G_i$ an edge chosen uniformly at random from the non-present ones, for $0\leq i<N$. Equivalently let $e_1,e_2,...,e_N$ be a permutation of the edges of the complete graph $K_n$ chosen uniformly at random and set $G_i=([n],\{e_1,...,e_i\})$, $0\leq i<N$. Let $\tau_2$ be the minimum $i$ such that $G_i$ has minimum degree $2$  and $\tau_H$ be the minimum $i$ such that $G_i$ is Hamiltonian. Bollob\'as \cite{bollobas1984} and independently Ajtai, Koml\'os and Szemer\'edi \cite{ajtai1985},  building upon work of P\'osa \cite{posa1976} and Korshunov \cite{korshunov}, proved that w.h.p.\footnote{We say that a sequence of events $\{\mathcal{E}_n\}_{n\geq 1}$ holds {\em{with high probability}} if $\lim_{n \to \infty}\Pr(\mathcal{E}_n)=1-o(1)$.} $\tau_2=\tau_H$. Thus, to achieve Hamiltonicy, one has to wait until the minimum degree becomes 2. Unfortunately, this necessary condition, is satisfied only by relatively dense graphs w.h.p. which have at least $0.5n\log n$ while a Hamilton cycle uses only $n$ of them. This raises the following question. Can one built a Hamiltonian subgraph of $G_t$ that spans $(1+o(1))n$ edges in an on-line fashion for some $t=(1+o(1))\tau_2$?

A generalization of this question was studied by Frieze, Krivelevich and Michaeli in the following setting that they introduced in \cite{frieze2022fast}. Once again let $e_1,e_2,...,e_N$ be a permutation of $E(K_n)$ chosen uniformly at random. The sequence $e_1,e_2,...,e_N$ is revealed, one edge at a time. Starting with the empty graph on $[n]$, as soon as an edge is revealed we must decide, immediately and irrevocably, whether to choose and add it to our graph. Let $B_i$ be the graph constructed after the $i$th edge is revealed. Let $\cB_{HAM}'$ be the set of pairs $(t,b)$ for which there exists an algorithm that builds a Hamiltonian graph of size at most $b$ within the first $t$ rounds of the above process w.h.p.
Clearly, as $B_i\subseteq G_i$ for all $i$ and $\tau_2>0.5n\log n$ w.h.p., a necessary condition for $(t,b)\in \cB_{HAM}$ is that $t\geq 0.5n\log n$ and $b\geq n$. Frieze, Krivelevich and Michaeli proved that for every $\epsilon>0$ there exists $C>0$ such that if $t\geq (0.5+\epsilon)n\log n$ and $b\geq 9n$  or $t\geq Cn\log n$ and $b\geq (1+\epsilon)n$ then $(t,b)\in \cB_{HAM}'$. 
They also conjectured that exists $\epsilon>0$ such that  $t\leq (0.5 +\epsilon)n\log n$ and $b\leq (1+\epsilon)n$ and $(t,b)\notin \cB_{HAM}'$ (see Conjecture $9$ of \cite{frieze2022fast}). Theorem \ref{thm:hamrestricted} together with Remark \ref{rem:extensions} refute this conjecture. 

A second way to generalize our question is within the framework of the Achlioptas processes. Inspired by the ``power of two choices" paradigm Achiloptas proposed the following process. Starting with the empty graph on $[n]$, at each round a set of $K=K(n)$ edges is presented chosen uniformly at random from the missing ones, and we are asked to choose one of them to add it to the current graph, immediately and irrevocably. By taking $K=1$ or choosing a uniformly random edge at each round (and ignoring repetitions of edges) one retrieves the random graph process. The aim of the Achlioptas process is to accelerate or delay a given graph property. For example Bohman and Frieze proved that there exist $\e>0$ and a strategy that w.h.p. ensure that one can construct a graph with no component of size  $\Omega(n)$ after $(1+\epsilon)n/2$ rounds, thus delaying the appearance of the giant \cite{bohman2001avoiding}. Krivelevich, Lubetzky and Sudakov studied $\tau_H(K)'$, the minimum number of rounds that one needs in order to construct a Hamiltonian graph in the above process \cite{krivelevich2010hamiltonicity}. They proved that w.h.p.
\begin{equation}\label{eq:bounds}
    (1+o(1))\bigg(1+\frac{\log n}{2K}\bigg)n \leq \tau_H(K) \leq (1+o(1))\bigg(3+\frac{\log n}{K}\bigg)n.
\end{equation}
They obtained the upper bound by constructing a random $3$-out graph which is known to be Hamiltonian \cite{bohman2009hamilton}. For the lower bound they proved that for any algorithm $\mathcal{A}$ and any $\epsilon>0$, after $(1-\epsilon)(0.5d+0.5\log n/K)n$ rounds there remain $n^{\epsilon/2}$ vertices of degree smaller than $d$ w.h.p. Their argument goes as follows. After $0.5(1-\epsilon)dn$ rounds the graph constructed so far by $\mathcal{A}$ contains at least $\epsilon d n$ vertices of degree smaller than $d$, deterministically. From those vertices, at least $n^{\epsilon/2}$ will not be incident to any edge that will be presented in the next $0.5(1-\epsilon)n(\log n)/K$ rounds. 
Krivelevich, Lubetzky and Sudakov also proved that the lower bound in 
\ref{eq:bounds} is the correct one in the regimes $K=o(\log n)$ and $K=\omega(\log n)$. In these regimes the lower bound reduces to $(1+o(1))(\log n)/2K)n$ and $(1+o(1))n$ respectively. Theorem \ref{thm:hamrestricted} implies that the lower bound is always the correct one. The problem of improving the bounds in \eqref{eq:bounds} is also stated as Problem 43 in the bibliography on Hamilton cycles in random graphs by Frieze \cite{frieze2019}.

Formally the process that we consider is the following one. Starting with the empty graph on $[n]$, at each round a set of $K=K(n)$ edges is presented chosen uniformly at random from the missing ones, and we are asked to chose \emph{at most} of them and add it to the current graph, immediately and irrevocably. We let $B_i$ be the graph constructed after $i$ rounds. 
We let $\cB_{HAM}$ be the set of pairs $(t,b)$ for which there exists an  algorithm that builds a Hamiltonian graph of size $b$ within the first $t$ rounds of the above process w.h.p. Similarly, we let $\cB_{PM}$ be the set of pairs $(t,b)$ for which there exists an algorithm that builds a matching of size $\lfloor n/2 \rfloor$ within the first $t$ rounds of the above process w.h.p.
\begin{theorem}\label{thm:hamrestricted}
Let $K=K(n)=O(\log n)$. Then,
$$\bigg(\bigg(1+\frac{250}{\log\log n}\bigg)\bigg(1+\frac{\log n}{2K}\bigg)n, \bigg(1+\frac{11}{\log\log n}\bigg)n\bigg) \in \cB_{HAM}.$$ 
\end{theorem}

The case $K=\omega(\log n)$ of the above Theorem follows from Theorem 1.2 of \cite{krivelevich2010hamiltonicity}. Once again, as $G_t$ has minimum degree $0$ for $t\leq 0.5n\log n$ w.h.p. one has that $(t,b)\in \cB_{PM}$ only if $t\geq 0.5n\log n$ and $b\geq n/2$.
\begin{theorem}\label{thm:pmrestricted}
$$\bigg(\bigg(1+\frac{250}{\log\log n}\bigg)\bigg(0.5+\frac{\log n}{2K}\bigg)n, \bigg(0.5+\frac{11}{\log\log n}\bigg)n\bigg) \in \cB_{PM}.$$ 
\end{theorem}

\begin{remark}\label{rem:extensions}
One may consider the variation of the process where at every round $K$ edges are presented chosen uniformly at random from the ones that have not been presented yet. Theorems \ref{thm:hamrestricted} and \ref{thm:pmrestricted} as stated also hold for this variation.
\end{remark}

The rest of the paper is organized as follows. At Section 2 we review some properties of the strong $4$-core of the binomial random graph. At Section 3 we prove Theorem \ref{thm:hamrestricted}. Finally at Section 4 we discuss the proof of Theorem \ref{thm:pmrestricted} and some extensions of Theorem \ref{thm:hamrestricted}. At various points at Section 3 we will use the Chernoff bounds as stated at the following lemma.
\begin{lemma}\label{lem:chernoff}
For all $\e\in [0,1]$,
$$\Pr(|Bin(n,p)-np|>\e np )\leq 2e^{-\frac{\e^2 np}{3}}.$$
\end{lemma}

\section{The strong $k$-core}
For a graph $G$ we define the \emph{strong $k$-core} of $G$ to be the  maximal subset $S$ of $V(G)$ with the property that every vertex in $S\cup N(S)$ has at least $k$ neighbors in $S$. Observe that if the sets $S_1,S_2\subset V(G)$ have this property then so does the set $S_1\cup S_2$. Thus the strong $k$-core of a graph is well-defined. It also naturally partitions the vertex set of a graph $G$ into 3 sets which we denote by $V_{k,black}(G), V_{k,blue}(G)$ and $V_{k,red}(G)$ where $V_{k,black}(G)$ is the strong $k$-core of $G$, $V_{k,blue}(G)$ its neighborhood and $V_{k,red}(G)$ is the rest i.e. $V_{k,red}(G)=V(G)\setminus (V_{k,black}(G)\cup N(V_{k,black}(G))$. In our knowledge the strong $3$-core was first used in \cite{anastos2021scaling} for finding the longest cycle in sparse random graphs while the concept of the strong $k$-core was first formalised in \cite{anastos2021note}. There it was observed that the strong $4$-core of $G(n,c/n)$ is robustly Hamiltonian for $c\geq 20$ as described at the following theorem. For a graph $G$ and $U\subseteq V(G)$ denote by $G[U]$ the subgraph of $G$ induced by $U$. By $G(n,p)$ we denote the binomial random graph i.e. the random graph on $[n]$ where every edge appears independently with probability $p$. 
\begin{theorem}[Theorem 3.3 of \cite{anastos2021note}]\label{thm:ham}
Let $G\sim G(n,c/n)$, $c\geq 20$. Let $G'$ be the subgraph of $G$ induced by $V_{4,black}(G)\cup V_{4,blue}(G)$. Then for every $U \subseteq V_{4,blue}(G)$ and matching $M$ on  $V_{4,blue}\setminus U$  with probability $1-o(n^{-1})$ we have that  $G'[V(G)\setminus U]\cup M$ has a Hamilton cycle that spans $M$. 
\end{theorem}
The above theorem reduces the problem of finding a Hamilton cycle in a random graph $G$ to finding a subgraph $G'$ of $G$ on $n'$ vertices that is distributed as $G(n',20/n')$ for some and then covering $V(G)\setminus V_{4,black}(G')$ by vertex disjoint paths whose endpoints lie in $V_{4,blue}(G')$. The next lemma states that w.h.p. $V_{4,blue}(G')$ is linear in $|V(G')|$. For its proof see part (c) of Lemma 3.3 of \cite{anastos2021note} and  the proof of part (b) of Theorem 1.1. of \cite{anastos2021note}. 
\begin{lemma}\label{lem:red}
Let $G\sim G(n,dn)$, $d\geq 10$. Then  $|V_{4,blue}(G)|\geq 0.1\cdot (2d)^3e^{-2d}n$ with probability $1-o(n^{-1})$.
\end{lemma}

\section{Constructing a Hamilton cycle fast on-line efficiently}
In this section we prove Theorem \ref{thm:hamrestricted}. Our algorithm   consists of $5$-phases. Let $n'= n/\log\log n$. The first phase aims to construct a subgraph $G'$ on $n'$ vertices whose strong $4$-core has the robust Hamiltonicity property described by Theorem \ref{thm:ham}. The rest of the phases aim to cover the vertices in $U=[n]\setminus (V_{4,black}(G')\cup V_{4,blue}(G'))$ by vertex disjoint paths with endpoints in $V_{4,blue}(G')$. To do so we first greedily cover most of the vertices in $U$ by at most  $n/(\log\log n)^2$ vertex disjoint paths each of length at most $\log n$. Here we allow paths of length $0$ corresponding to single vertices. Let $\mathcal{P}$ be the set of these paths. Then, at Phase $3$ we greedily match the endpoints of every path $P\in \mathcal{P}$ to  $V_{4,blue}(G')$ and extent it into a path whose endpoints belong to $V_{4,blue}(G')$. Failing this, at Phase 4, we attempt to connect the endpoints of $P$ to many vertices in the interior of other paths in $\mathcal{P}$. Finally at Phase 5 using the edges selected at Phase 4 we reroute $P$ through the rest of the paths. Such a rerouting may looks as follows. Let $Q=v_1,v_2,...,v_k$ and $P=u_1,u_2,...,u_r$ be vertex disjoint paths with $v_1,v_k,u_r\in V_{4,blue}(G')$ and $u_1\notin V_{4,blue}(G')$. In such a case, adding the edges $u_1v_i$ and $v_{i-1}v$  with $v\in V_{4,blue}(G)$, $2<i\leq k$ (selected at phases 4 and 5 respectively) and removing the edge $v_{i-1}v_i$ from $E(P)\cup E(Q)$  results to 2 vertex disjoint paths that cover $V(P)\cup V(Q)$ and have all of their endpoints in $V_{4,blue}$. The above procedure builds a graph which contains a set of paths $\mathcal{P}$ with endpoints in $V_{4,blue}(G')$ that cover $U$  and are spanned by $U\cup V_{4,blue}(G')$. Now contract each of the paths in $\mathcal{P}$ into a single edge and invoke Theorem \ref{thm:ham} to identify the desired Hamilton cycle. 

We now proceed with the description of the $5$ phases and their analysis.  We let $t_\e=50(1+\log n/2K)(n/\log \log n)$, $t_1=t_\e$, $t_2=t_1+t_\e+n$, $t_3=t_2+t_\e$, $t_4=t_3+t_\e+n(\log n/2K)$ and $t_5=t_4+t_\e$. For $i=1,...,4$ if Phase $i$ returns FAILURE then we terminate our algorithm. Else we proceed and execute phase $i+1$.
\begin{algorithm}[H]
\caption{Phase 1}
\begin{algorithmic}[1]
\For{ $t=1$ to $t_1$, at round $t$} 
\\ \hspace{5mm} If $|E(B_{t-1})|< 10n'$ the select the first edge that is spanned by $[n']$.
\EndFor
\\Let $G'$ be the graph on $[n']$ spanned by the selected edges. 
\\ If $|E(G')|\leq 10n'$ or  $|V_{4,blue}(G')|\leq  10^{-6} n'$ the return FAILURE. 
\end{algorithmic}
\end{algorithm} 
The probability that $|E(B_{t-1})|<10n'$ and at a round $t$ we are  presented with an edge spanned by $[n']$ is at least $1-(1- (\binom{n'}{2}-10n')/\binom{n}{2})^K= 1-e^{-(1+o(1))K(n'/n)^2}= 1-e^{-(1+o(1))K/( \log\log n)^2}$ 
The calculation below and Lemma \ref{lem:red} imply that w.h.p. the algorithm does not return FAILURE at the end of Phase 1. 
$$\Pr(|E(G')|< 10n') \leq \Pr\bigg(Bin\bigg(t_\e,1-e^{-\frac{(1+o(1))K}{ (\log\log n)^2}}\bigg) <10n'\bigg)=o(1). $$
The last equality  follows from the Chernoff bounds and the following case distinction. If $K\leq (\log\log n)^2$ then
$$t_\e \bigg(1-e^{-\frac{(1+o(1))K}{ (\log\log n)^2}}\bigg) \geq t_\e  \cdot \frac{0.6K}{ (\log\log n)^2} \geq \frac{n\log n}{2K\log\log n}  \cdot \frac{0.6K}{ (\log\log n)^2}=\omega(n).$$
Else if $  (\log\log n)^2 \leq K =O(\log n)$ then $t_\e (1-e^{-\frac{(1+o(1))K}{(\log\log n)^2}})\geq 0.6 t_\e >10n'.$

Let  $Z=V_{4,black}(G)\cup V_{4,blue}(G')$, $U=[n]\setminus Z$ and  $W=V_{4,blue}(G')$. 

\begin{algorithm}[H]
\caption{Phase 2}
\begin{algorithmic}[1]
\For{$t=t_1+1$ to $t_2$, at round $t$}
\\ \hspace{5mm} Select the first edge that is incident to 2 vertices $u,v$ in $U$ of degree at most $1$ such that (i) none of $u,v$  is an endpoint of a path of length at least $0.5\log n$ in $B_{t-1}$ and (ii) $u,v$ are not the endpoints of the same path in $B_{t-1}$.
\EndFor
\\Let $\mathcal{P}$ be the maximal set of paths spanned by $E_2=E(B_{t_2})\setminus E(B_{t_1})$. Here we allow paths of length $0$ corresponding to single vertices. So if some vertex $v$ is not incident to $E_2$ then $v$ belongs to $\cP$. 
\\If $|\mathcal{P}|\leq n/(\log\log n)^2$ then return FAILURE. 
\end{algorithmic}
\end{algorithm} 

There are at most $n/2\log n$ paths of length at least $0.5\log n$. Thus while $U$ cannot be covered by $n/(\log\log n)^2$ paths spanned by edges selected during Phase 2, at least $0.9 n/(\log\log n)^2$ of these paths have length at most $0.5\log n-1$. If any of the at least  $0.8n^2/(\log\log n)^4$ edges joining a pair of endpoints of distinct such paths is presented then an edge is selected and the number of paths will be reduced by 1. This occurs with probability at least $1-(1-(1.6+o(1))/(\log\log n)^4)^K\geq 1-e^{-\frac{K}{(\log\log n)^4}}.$
Therefore, Phase 2 returns FAILURE with probability at most 
\begin{equation*}
    \Pr\bigg(Bin\bigg(n+t_\e, 1-e^{-\frac{K}{(\log\log n)^4}}\bigg)<n\bigg)=o(1).
\end{equation*}
The above calculation follows from the Chernoff bounds and the following case distinction. If $K\leq (\log\log n)^4$ then, as $t_\e =50(1+\log n/2K)(n/\log \log n)$,
\begin{align*}
&    (n+t_\e)\bigg(1-e^{-\frac{K}{(\log\log n)^4}}\bigg) \geq \bfrac{n\log n}{2K\log\log n}\bfrac{0.6K}{(\log\log n)^4}  =\omega(n). 
\end{align*}
Else if $  (\log\log n)^4 \leq K \leq \log^{0.9}n$ then,
\begin{align*}
&    (n+t_\e)\bigg(1-e^{-\frac{K}{(\log\log n)^4}}\bigg)\geq n(1+0.5\log^{0.09}n)(1-e^{-0.1}) =\omega(n). 
\end{align*}
Else $  \log^{0.9}n \leq K =O(\log n)$ and 
\begin{align*}
&    (n+t_\e)\bigg(1-e^{-\frac{K}{(\log\log n)^4}}\bigg)\geq \bigg(n+ \frac{n}{\log\log n}\bigg)\bigg(1-\frac{1}{(\log\log n)^2}\bigg)
\geq n\bigg(1+ \frac{0.5}{\log\log n}\bigg).
\end{align*}

Let $End$ be the multiset consisting of endpoints of $\mathcal{P}$. Every element of $End$ has multiplicity $1$ except the elements that correspond to paths of length $0$ in $\mathcal{P}$ which have multiplicity $2$.
\begin{algorithm}[H]
\caption{Phase 3}
\begin{algorithmic}[1]
\For{$t=t_2+1$ to $t_3$, at round $t$}
\\\hspace{5mm} Select the first edge $e$  that 
matches a vertex in $End$ to a  vertex in $W$  and remove a sole copy of its endpoints from $End$ and $W$ respectively.
\EndFor
\\ Update $\mathcal{P}$ to be the maximal set of paths spanned by $E(B_{t_3})\setminus E(B_{t_1})$.
\\ If $|End|> n/\log^5 n$ then return FAILURE.
\end{algorithmic}
\end{algorithm} 
If Phases 1 and 2 do not return FAILURE then throughout Phase 3
$$|W|\geq |V_{4,blue}(G')|-2|\mathcal{P}|\geq  10^{-6}n'-2n/(\log\log n)^2\geq 10^{-7}n'.$$ In addition initially $|End|\leq 2|\mathcal{P}|\leq 4n/(\log\log n)^2$. Phase 3 returns FAILURE only if there exists $0\leq i \leq 6\log\log n$ such that at the beginning of round $t_2+1+(1-2^{-i})t_\e$ the set $End$ has size at most $2^{2-i}n/(\log\log n)^2$ while at the end of round $t_2+(1-2^{-i-1})t_\e$ the set $End$ has size larger than $2^{1-i}n/(\log\log n)^2$. This occurs with probability at most
\begin{align*}
    &\sum_{i=0}^{6\log\log n} \Pr\bigg(Bin\bigg(2^{-i-1}t_\e, 1-\bigg(1-\frac{10^{-7}n' \cdot \frac{2^{1-i}}{(\log\log n)^2}}{\binom{n}{2}}\bigg)^K\bigg)< \frac{2^{1-i}n}{(\log\log n)^2}\bigg)
    \\& \leq \sum_{i=0}^{6\log\log n} \Pr\bigg(Bin\bigg(2^{-i-1}t_\e, 1-e^{-\frac{2^{-i}K}{(\log\log n)^4}}\bigg)< \frac{2^{-i}n}{(\log\log n)^2}\bigg)=o(1).
\end{align*}
The last equality follows from the Chernoff bounds and the following case distinction. If $2^{-i}K <(\log\log n)^4$ then 
\begin{align*}
    &t_\e\bigg(1-e^{-\frac{2^{-i}K}{(\log\log n)^4}}\bigg) 
\geq t_\e \cdot \frac{0.6\cdot 2^{-i}K}{(\log\log n)^4}
\geq \frac{n\log n}{2K} \cdot \frac{0.6\cdot 2^{-i}K}{(\log\log n)^4} \geq 
2^{-i} n.
\end{align*}   
Else $2^{-i}K \geq (\log\log n)^4$ and
\begin{align*}
    &t_\e\bigg(1-e^{-\frac{2^{-i}K}{(\log\log n)^4}}\bigg)\geq\frac{n}{\log\log n}  (1 -e^{-1} ) = \geq \bfrac{0.1n}{\log\log n}.
\end{align*}   
\begin{algorithm}[H]
\caption{Phase 4}
\begin{algorithmic}[1]
\\Orient the paths in $\mathcal{P}$ and let $\mathcal{P}^+$ be the set of paths in $\mathcal{P}$ with no endpoint in $End$.
\\ For $v\in End$ set $End_v=\mathcal{P}_v=\emptyset$.
\For{ $t=t_3+1$ to $t_4$ at round $t$}
\\ \hspace{5mm} Select the first edge  $e=vw$ such that $v\in End$, $|End_v|<\log^{0.8}n$ and $w$ belongs to some path $P \in \mathcal{P}^+$ and  is not the first vertex on $P$, add the vertex preceding $w$ on $P$ to $End_v$, add $P$ to $\mathcal{P}_v$ and remove $P$ from $\mathcal{P}^+$.
\EndFor
\\if there exists $v\in End$ such that $|End_v| \leq \log^{0.8} n$ then return FAILURE. 
\end{algorithmic}
\end{algorithm} 

Each path in $\mathcal{P}$ has size at most $\log n$. In addition $\mathcal{P}$ covers $U$, has size $o(n')$ and $|End|\leq  n/\log^5 n$. Hence for $v\in End$ at each round there exist at least $n-20n'-|\mathcal{P}|-|End|(\log n+\log^{0.8}n)\geq (1-22(\log\log n)^{-1})n$ vertices $w$ such that if the edge $e=vw$ is presented, no other edge incident to $End$ is presented and $|End_v|<\log^{0.8}n$ then $|End_v|$ is increased by $1$.  This occurs with probability at least 
$$K \cdot  \frac{2(1-22(\log\log n)^{-1})}{n} 
\bigg(1-\frac{n|End| }{\binom{n}{2}}\bigg)^{K-1} \geq \frac{2K(1-30(\log\log n)^{-1})}{n}.
$$
Therefore, with $p=\bigg(1-\frac{30}{\log\log n}\bigg)\frac{2K}{n}$ and $r=t_\e+\frac{n\log n}{2K} > \big(1+\frac{50}{\log\log n}\big)\frac{n\log n}{2K}$, and the probability that Phase 4 return FAILURE is at most  
\begin{align*}
 &n\Pr(Bin(r,p)\leq \log^{0.8}n) \leq 2n \binom{r}{\log^{0.8}n} p^{\log^{0.8}n}(1-p)^{r-\log^{0.8} n} 
\\& \leq 2n \bfrac{erp}{\log^{0.8}n}^{\log^{0.8}n} e^{-pr+ p\log^{0.8}n}
\leq  2n (2\log n)^{0.2\log^{0.8}n} e^{-\log n - \frac{\log n}{\log\log n}}
=o(1).
\end{align*}

\begin{algorithm}[H]
\caption{Phase 5}
\begin{algorithmic}[1]
\\ Let $E^+=E^-=\emptyset$
\For{ $t=t_4+1$ to $t_5$, at round $t$}
\\ \hspace{5mm} Select the first edge $e=xy$ such that  $x\in End_v$ for some $v\in End$ and $y\in W$, remove $y$ from $W$, add $e$ to $E^+$ and add the edge  $xx'$ to $E^-$. Here $x'$ is the vertex succeeding $x$ on the unique path in $\mathcal{P}$ that contains $x$.
\EndFor
\\ If $End \neq \emptyset$ then return FAILURE. 
\end{algorithmic}
\end{algorithm} 
For $v\in End$ there exists at least $|End_v|\times |W|$ pairs satisfy the condition described above. Once again, if the phases $1$ and $2$ do not return FAILURE then throughout the execution of Phase $5$ we have that $|W|\geq 10^{-6}n'-2n/(\log\log n)^2>10^{-7}n'$. Therefore,  
$|End_v|\times |W| \geq n\log^{0.5}n$. As $t_\e \geq n/\log\log n$ and $|End|<n$,  Phase 5 returns FAILURE with probability at most 
\begin{align*}
\Pr\bigg(Bin\bigg(\frac{n}{\log\log n},\frac{2\log^{0.5}n}{n}\bigg)<  n\bigg) =o(1).
\end{align*} 
Let $\mathcal{E}$ be the event that none of the 5 phases returns FAILURE. From the preceding analysis of phases 1 to 5 we have that $\Pr(\mathcal{E})=1-o(1)$.
\begin{lemma}\label{lem:countedges}
The above algorithm selects at most $11n'+n$ edges in total. In addition, if $\mathcal{E}$ occurs  then the selected edges span a set of vertex disjoint paths $\mathcal{P}^*$ that cover $U$ such that for $P\in \mathcal{P}^*$ the  endpoints of $P$ belong to $V_{4,blue}(G')$ and the rest of the vertices in $V(P)$ belong to $U$. 
\end{lemma}

\begin{proof}
 Let $E_i=E(B_{t_i})\setminus E(B_{t_{i-1}})$ for $i=1,...,5$.  $E_2$ induces a set of paths on $[n]$ while $|E_3|+|E_5|\leq 2\mathcal{P}=o(n')$. Thereafter $|E_1|= 10n'$ and $|E_4|\leq |End|\log^{0.8}n=o(n')$. Thus, in the event $\mathcal{E}$ at most $11n'+n$ edges are selected in total. 

Thereafter, assume  that $\mathcal{E}$ occurs  and let $E_P= E_2\cup E_3\cup E^+)\setminus E^-$. Then due construction the set $E_P$ does not span a cycle. In addition every vertex in $U,V_{4,blue}$ and $V_{4,black}(G)$ respectively in incident to $2$, at most $1$ and $0$ respectively edges in $E_P$. Therefore $E_P$ induces the desired set of paths.  
\end{proof}

\textbf{Proof of Theorem \ref{thm:hamrestricted}} Assume that the high probability event $\mathcal{E}$ occurs. For every path $P\in \mathcal{P^*}$ let $e_P$ be an edge joining its endpoints. Then the set $M=\{e_P:P\in \mathcal{P^*}\}$ induces a matching on $V_{4,blue}(G')$. Let $G^*$ be the subgraph of $B_{t_5}$ spanned by $V_{4,blue}(G')\cup V_{4,black}(G')$. Theorem \ref{thm:ham} implies that there exists a Hamilton cycle $H$ in $G^*\cup M$ that spans all the edges of $M$. Substituting every edge $e_P\in M$ on $H$ with the path $P$ gives a Hamilton cycle in $B_{t_5}$. As $|E(B_{t_5})|\leq n+11n'$ by Lemma \ref{lem:countedges} and $t_5< (1+250/\log\log n)(1+(\log n)/2K)n$ the statement of Theorem \ref{thm:hamrestricted} follows.

\section{Concluding Remarks}
1. To proof Theorem \ref{thm:pmrestricted} one may use a 5-phase algorithm as done in the proof of Theorem \ref{thm:hamrestricted}. First execute Algorithm 1 and let $Z=V_{4,black}\cup V_{4,blue},W=V_{4,blue}$ and $U=[n]\setminus Z$. At Phase 2 greedily match the vertices in $U$ and at Phase 3 greedily match the unmatched vertices in $U$ to vertices in $W$. Let $End$ be the set of unmatched vertices and $M$ be the current matching. At Phase 4 match all the vertices in $End$ to $\log^{0.8}n$ matched vertices in $U$. Finally at Phase 5, by considering the edges selected at Phase 4, select edges that create $M$-augmenting paths of length 3 from $End$ to $W$ and augment along those paths until every vertex in $U$ is matched. At this point there exists $W'\subset W$ such that $M$ pairs up the vertices in $W'\cup U$. By Theorem \ref{thm:ham} the graph spanned by $Z\setminus W'$ has a Hamilton cycle $H$. Let $M'$ be a maximum matching spanned by $H$. Then $|M\cup M'|=\lfloor n/2\rfloor$ as it  matches all but at most $1$ vertices of $G$.

2. Let $\epsilon>0$. In place of Theorem \ref{thm:hamrestricted} one may asks what is the minimum number of rounds $\tau_{k-HAM}(K)$ needed to construct a graph that spans $k\in \mathbb{N}^+$ edge disjoint Hamilton cycles under the restriction of selecting at most $(k+\epsilon)n$ edges in total. One can construct such a graph by taking the union of $k$ Hamiltonian graphs. These can be constructed by executing our algorithm $k$ times. These $k$ execution will share Phase $4$. Thus $\tau_{k-HAM}(K)\leq (1+o(1))(0.5k+0.5\log n/K)n$. Once again this bound matches the lower bound on the number of rounds needed to construct a graph of minimum degree $k$ due to Krivelevich, Lubetzky and Sudakov.
 
3. Finally one may ask whether a hitting time version of Theorem \ref{thm:hamrestricted} holds. More concretely, let $\epsilon>0$ and $\tau_2$ be minimum such that every vertex is incident to at least $2$ edges that are presented during the first $\tau_2$ rounds. Does there exists an algorithm that constructs a Hamiltonian graph by selecting at most $(1+\epsilon)n$ edges during the first $\tau_2$ rounds w.h.p.? In contrast with the hitting time result regarding the Hamiltonicity of the random graph process and the somewhat tight Theorem \ref{thm:hamrestricted}, the answer to this question is negative, assuming that $n-K(n)=\Omega(n)$. Indeed, let $\epsilon<10^{-9}$ and fix an algorithm $\mathcal{A}$ that selects at most $(1+\epsilon)n$ edges during the first $\tau_2$ rounds. Let $B'$ and $B$ be the graph constructed by $\mathcal{A}$ by the end of rounds $\tau_2-1$ and $\tau_2$ respectively. In not hard to argue that w.h.p. either $B$ is not Hamiltonian or $B'$
$B'$ has a single vertex of degree $1$, say $v$. Furthermore, with probability bounded away from $0$, $v$ is incident to a single edge presented at round $\tau_2$ and the other endpoint of that edge has exactly $2$ neighbors in $B'$ both having degree $2$. Thus $B$ contains $3$ vertices of degree $2$ with a common neighbor, this is an obstruction for having a Hamilton cycle.
\subsection*{Acknowledgment} This project has received funding from the European Union’s Horizon 2020 research and innovation
programme under the Marie Sk\l{}odowska-Curie grant agreement No 101034413
\includegraphics[width=5.5mm, height=4mm]{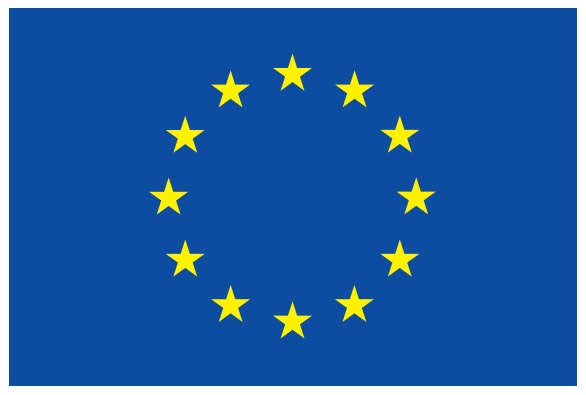}.
\bibliographystyle{plain}
\bibliography{bib}

\begin{thebibliography}{10}

\bibitem{ajtai1985}
Mikl{\'o}s Ajtai, J{\'a}nos Koml{\'o}s, and E~Szeraer{\'e}di.
\newblock First occurrence of hamilton cycles in random graphs.
\newblock {\em North-Holland Mathematics Studies}, 115(C):173--178, 1985.

\bibitem{anastos2021note}
Michael Anastos.
\newblock A note on long cycles in sparse random graphs.
\newblock {\em arXiv preprint arXiv:2105.13828}, 2021.

\bibitem{anastos2021scaling}
Michael Anastos and Alan Frieze.
\newblock A scaling limit for the length of the longest cycle in a sparse
  random graph.
\newblock {\em Journal of Combinatorial Theory, Series B}, 148:184--208, 2021.

\bibitem{bohman2001avoiding}
Tom Bohman and Alan Frieze.
\newblock Avoiding a giant component.
\newblock {\em Random Structures \& Algorithms}, 19(1):75--85, 2001.

\bibitem{bohman2009hamilton}
Tom Bohman and Alan Frieze.
\newblock Hamilton cycles in 3-out.
\newblock {\em Random Structures \& Algorithms}, 35(4):393--417, 2009.

\bibitem{bollobas1984}
B{\'e}la Bollob{\'a}s.
\newblock The evolution of sparse graphs, graph theory and combinatorics, 1984.
\newblock {\em MR}, 777163(2):35--57, 1984.

\bibitem{frieze2019}
Alan Frieze.
\newblock Hamilton cycles in random graphs: a bibliography.
\newblock {\em arXiv preprint arXiv:1901.07139}, 2019.

\bibitem{frieze2022fast}
Alan Frieze, Michael Krivelevich, and Peleg Michaeli.
\newblock Fast construction on a restricted budget.
\newblock {\em arXiv preprint arXiv:2207.07251}, 2022.

\bibitem{korshunov}
Aleksei~Dmitrievich Korshunov.
\newblock Solution of a problem of erd{\H{o}}s and renyi on hamiltonian cycles
  in nonoriented graphs.
\newblock In {\em Doklady Akademii Nauk}, volume 228, pages 529--532. Russian
  Academy of Sciences, 1976.

\bibitem{krivelevich2010hamiltonicity}
Michael Krivelevich, Eyal Lubetzky, and Benny Sudakov.
\newblock Hamiltonicity thresholds in achlioptas processes.
\newblock {\em Random Structures \& Algorithms}, 37(1):1--24, 2010.

\bibitem{posa1976}
Lajos P{\'o}sa.
\newblock Hamiltonian circuits in random graphs.
\newblock {\em Discrete Mathematics}, 14(4):359--364, 1976.

\end{thebibliography}
\end{document}